\documentclass{amsart}

\usepackage[english]{babel}

\usepackage{fullpage}

\usepackage{amstext,amsmath,amsthm,amssymb,mathrsfs}
\usepackage{amscd}
\usepackage{verbatim}
\usepackage{float}

\usepackage{bbm}

\usepackage{mathtools}
\usepackage{esint}

\usepackage{faktor}

\usepackage{enumerate}

\usepackage{appendix}

\usepackage{bm}

\DeclareMathOperator*{\foo}{\scalerel*{+}{\sum}}

\usepackage{scalerel}

\usepackage[colorinlistoftodos,prependcaption,textsize=tiny]{todonotes}

\allowdisplaybreaks

\DeclareMathAlphabet{\mathpzc}{OT1}{pzc}{m}{it}

\newcommand{\C}{\mathbb{C}}
\newcommand{\R}{\mathbb{R}}
\newcommand{\N}{\mathbb{N}}
\newcommand{\F}{\mathbb{F}}

\newcommand{\mbs}{\boldsymbol}

\newcommand{\mc}{\mathcal}
\newcommand{\ms}{\mathscr}
\newcommand{\mz}{\mathpzc}

\newcommand{\g}{\gamma}

\newcommand{\Longlra}{\ensuremath{\Longleftrightarrow}}

\newcommand{\longra}{\ensuremath{\longrightarrow}}

\renewcommand{\vec}[1]{\bm{#1}}

\newcommand{\one}{\mathbbm{1}}

\newcommand{\supp}{\text{\rm supp\,}}

\DeclarePairedDelimiter\abs{\lvert}{\rvert}
\newcommand{\hab}[1]{\bigl(#1\bigr)}
\newcommand{\has}[1]{\Bigl(#1\Bigr)}
\newcommand{\norm}[1]{||#1||}

\newcommand{\normB}[1]{\Big|\Big|#1\Big|\Big|}
\newcommand{\normb}[1]{\big|\big|#1\big|\big|}
\newcommand{\normm}[1]{\left|\left|\left|#1\right|\right|\right|}
\newcommand{\ip}[2]{\langle #1, #2 \rangle}

\newcommand{\ud}{\,\mathrm{d}}

\newcommand{\loc}{{\rm loc}}

\theoremstyle{plain}
\newtheorem{thm}{Theorem}[section]
\newtheorem{theorem}[thm]{Theorem}
\newtheorem{lemma}[thm]{Lemma}

\newtheorem{proposition}[thm]{Proposition}
\newtheorem{cor}[thm]{Corollary}
\newtheorem{corollary}[thm]{Corollary}

\theoremstyle{definition}

\theoremstyle{remark}
\newtheorem{remark}[thm]{Remark}

\usepackage{tikz}                           
\usetikzlibrary{shapes.geometric, arrows}

\begin{document}

\title[]{Pointwise Multiplication by the Characteristic Function of the Half-space on Anisotropic Vector-valued Function Spaces}

\author{Nick Lindemulder}
\address{Delft Institute of Applied Mathematics\\
Delft University of Technology \\ P.O. Box 5031\\ 2600 GA Delft\\The
Netherlands}
\address{Institute of Analysis \\
Karlsruhe Institute of Technology \\
Englerstra{\ss}e 2 \\
76131 Karlsruhe\\
Germany}
\email{nick.lindemulder@kit.edu}


\date{\today}

\begin{abstract}
We study the pointwise multiplier property of the characteristic function of the half-space on weighted mixed-norm anisotropic vector-valued function spaces of Bessel potential and Triebel-Lizorkin type.
\end{abstract}

\maketitle

\section{Introduction}\label{IR:sec:intro}
The motivation for this paper comes from the close connection between spaces of vanishing traces and the pointwise multiplier property of the
characteristic function of the underlying domain under consideration, and more particularly, the importance of this connection for interpolation for spaces with vanishing boundary conditions.
The first interpolation theorem for such spaces is due to Grisvard \cite{Grisvard1967} for $L_2$-based Sobolev-Slobodeckii spaces and the real interpolation functor,
which he subsequently extended to the $L_p$-based setting in \cite{Grisvard1969}. The corresponding result for the Sobolev/Bessel potential scale and the complex interpolation functor was subsequently obtained by Seeley \cite{Seeley1972}, but also see the more recent works \cite{Amann2019_LQPP_II} and \cite{Lindemulder&Meyries&Veraar2017} and the references given therein.

The real interpolation results from \cite{Grisvard1967,Grisvard1969} can be considered more elementary than the complex interpolation result from \cite{Seeley1972}. Whereas \cite{Grisvard1967,Grisvard1969} has the pointwise multiplier property of $\one_{\R^d_+}$ on the Sobolev-Slobodeckii space $W^{s}_{p}(\R^d)$ in the parameter range
\begin{equation}\label{eq:par-range_pm}
\frac{1}{p}-1 < s < \frac{1}{p}
\end{equation}
as a byproduct (also see \cite[Chapter~VIII]{Amann2019_LQPP_II}), \cite{Seeley1972} has as one of its main ingredients the result due to Shamir \cite{Shamir1962} and Strichartz \cite{Strichartz1967} that $\one_{\R^d_+}$ is a pointwise multiplier on the Bessel potential space $H^{s}_{p}(\R^d)$ in the parameter range \eqref{eq:par-range_pm}. 
For a further discussion on this we refer the reader to \cite[Chapter~VIII]{Amann2019_LQPP_II}, in particular, 
\cite[p.~341,~342,~368]{Amann2019_LQPP_II}.
Furthermore, extensions to Besov spaces $B^{s}_{p,q}(\R^d)$ and Triebel-Lizorkin spaces $F^{s}_{p,q}(\R^d)$ can be found in \cite{Franke1986,Marschall1987,Meyries&Veraar2015_pointwise_multiplication,Peetre1976,Sickel1987,Sickel1999,Sickel1999b,Triebel1983_TFS_I}, also see the monograph \cite{Runst&Sickel_monograph}.

Recently, Meyries and Veraar \cite{Meyries&Veraar2015_pointwise_multiplication} extended the pointwise multiplier result from \cite{Shamir1962,Strichartz1967}
to the vector-valued setting with Muckenhoupt weights. The weights considered are the power weights $w_\gamma$ given by 
\begin{equation}\label{eq:power_weight}
w_\gamma (t,x') = |t|^\gamma, \qquad t \in \R, x' \in \R^{d-1}, 
\end{equation}
for $\gamma \in (-1,p-1)$. It was shown that $\one_{\R^d_+}$ is a pointwise multiplier on the weighted vector-valued Bessel potential space $H^{s}_{p}(\R^d,w_\gamma,X)$ if $X$ is a UMD Banach space (see Section~\ref{subsec:UMD}) and $s$ satisfies
\begin{equation}\label{eq:weight_par-range}
\frac{1+\gamma}{p}-1 < s < \frac{1+\gamma}{p}, 
\end{equation}
where $\R^d_+=\{(t,x') \in \R \times \R^{d-1} : t > 0\}$.
Furthermore, analogous results were obtained for Besov and Triebel-Lizorkin spaces, without any restrictions on the Banach space $X$.

The motivation to consider the weighted vector-valued setting in \cite{Meyries&Veraar2015_pointwise_multiplication} is the
$L_p$-$L_q$-maximal regularity approach to parabolic evolution equations. In this setting the use of temporal and/or spatial weights of the form \eqref{eq:power_weight} allow one to treat rougher initial and/or boundary data, respectively, and they provide an inherent smoothing effect of the solutions, see for instance the papers \cite{Angenent1990,Clement&Simonett2001,Hummel&Lindemulder2019,Lindemulder2020_JEE,Lindemulder&Veraar2020_JDE,Meyries&Veraar2014_traces,Pruss&Simonett2004,PrSiZa} and the monograph \cite{Pruess&Simonett2016_book}.
Furthermore, temporal weights are of  fundamental importance in the recently developed theory of critical spaces and its applications to nonlinear partial differential equations of quasi- and semilinear type, see   \cite{Agresti&Veraar2020a,Agresti&Veraar2020b,Pruess2017,Pruess&Wilke2017,Pruess&Wilke2018,Pruess&Simonett&Wilke2018}. 

The pointwise multiplier property of $\one_{\R^d_+}$ on $H^{s}_{p}(\R^d,w_\gamma,X)$ from \cite{Meyries&Veraar2015_pointwise_multiplication} was first used by the author, Meyries and Veraar \cite{Lindemulder&Meyries&Veraar2017} to prove results on complex interpolation of weighted vector‐valued Sobolev/Bessel potential spaces on the half‐line with Dirichlet boundary conditions, with as an application the characterization of the fractional domain spaces of the first derivative operator on the half-line. It was subsequently used by Amann \cite{Amann2019_LQPP_II} to characterize complex interpolation spaces of Sobolev/Bessel potential spaces on half-spaces with normal boundary conditions. 
These interpolation results are (partial) extensions of Seeley's classical work in this direction \cite{Seeley1972}. 

In \cite{Amann2019_LQPP_II} the above mentioned interpolation result is \cite[Theorem~VIII.2.4.8]{Amann2019_LQPP_II}. Whereas \cite{Amann2019_LQPP_II} is mostly concerned with so-called \emph{anisotropic} function spaces, \cite[Theorem~VIII.2.4.8]{Amann2019_LQPP_II} restricts itself to the classical \emph{isotropic} setting. As explainded in \cite[Remark~VIII.2.4.9]{Amann2019_LQPP_II}, the reason for this restriction is the unavailability of an anisotropic version of the pointwise multiplier result from \cite{Meyries&Veraar2015_pointwise_multiplication}. A conjectured anisotropic analogue is left as an open problem on \cite[p.~342]{Amann2019_LQPP_II}. In this paper we will solve this open problem (see Theorem~\ref{thm:pm_H}) and, furthermore, obtain an analogous result for anisotropic Triebel-Lizorkin spaces (see Theorem~\ref{thm:pm_F}).
As an application one could obtain results on interpolation with boundary conditions analogously to \cite[Theorem~VIII.2.4.8]{Amann2019_LQPP_II}. However, we will leave this to the future.

Anisotropic function spaces naturally appear in the study of parabolic partial differential equations, where they provide a description of the corresponding parabolic regularity in which the time and space derivatives contribute in different strength, see e.g.\ \cite{Denk&Kaip2013,Hummel&Lindemulder2019,Koehne&Saal2020,Lindemulder2020_JEE,Lindemulder&Veraar2020_JDE}. Loosely speaking, there are two approaches to anisotropy in the context of function spaces. On the one hand, there is a Fourier analytic approach, where the anisotropic nature of the function space is obtained by using a suitable anisotropic scaling structure in the Fourier domain, see e.g.\ \cite{Amann2019_LQPP_II,Farkas&Johnsen&Sickel2000,Johnsen&Munch_Hansen&Sickel2014,Johnsen&Mucn_Hansen&Sickel2015,Johnsen&Sickel2007,Johnsen&Sickel2008,Lindemulder2019_IR,Lindemulder2020_JEE}. 
On the other hand, there is a functional analytic approach, where the anisotropic nature of the function space is obtained by taking intersections of function space-valued function spaces, see e.g.\ \cite{Denk&Hieber&Pruess2007,Koehne&Saal2020,Meyries&Schnaubelt2012_fractional_Sobolev,Meyries&Schnaubelt2012_maximal_regularity,Pruess&Simonett2016_book}. A bridge between the two approaches is provided by intersection representations, see e.g.\ \cite{Amann2019_LQPP_II,Lindemulder2019_IR}, which will play a major role in this paper.

In the maximal $L_p$-$L_q$-regularity problem for fully inhomogeneous parabolic boundary value problems, Triebel-Lizorkin spaces have turned out to naturally occur in the description of the sharp regularity of the boundary data (see \cite{Denk&Hieber&Pruess2007,Johnsen&Mucn_Hansen&Sickel2015,Johnsen&Sickel2008,Lindemulder2020_JEE,Weidemaier2002}).
For instance, in the special case of the heat equation with Dirichlet boundary condition, the boundary datum has to be in anisotropic mixed-norm Triebel-Lizorkin space
\begin{equation*}
F^{\delta,(\frac{1}{2},1)}_{(p,q),q}(J \times \partial\mathscr{O}) = F^{\delta}_{p,q}(J;L_q(\partial\mathscr{O})) \cap L_p(J;B^{2\delta}_{q,q}(\partial\mathscr{O})), \qquad \delta = 1-\frac{1}{2p},    
\end{equation*}
where the identification between these two spaces follows from the intersection representation \cite[Example~5.7]{Lindemulder2019_IR}.
Triebel-Lizorkin spaces can furthermore be important as a technical tool in the study of Sobolev and Bessel potential spaces, see e.g.\ \cite{Lindemulder2020_JEE,Lindemulder&Veraar2020_JDE,Meyries&Veraar2012_sharp_embedding,Meyries&Veraar2014_traces,Meyries&Veraar2015_pointwise_multiplication,Scharf&Schmeisser&Sickel_Traces_vector-valued_Sobolev,Schmeisser&Sickel2005,S&S_jena-notes} for the Banach space-valued setting. 
The latter even holds in the weighted setting with power weights $w_\gamma$ \eqref{eq:power_weight} with $\gamma$ outside the $A_p$-range $(-1,p-1)$, see \cite{Lindemulder&Veraar2020_JDE} and \cite{Lindemulder2018_DSOP,Lindemulder2020_DSOE,Hummel&Lindemulder2019}.

Motivated by \cite{Lindemulder2018_DSOP,Lindemulder2020_DSOE}, we will extend the corresponding results in \cite{Meyries&Veraar2015_pointwise_multiplication} for weighted Triebel-Lizorkin spaces beyond the $A_p$-setting considered there. Furthermore, motivated by the maximal $L_p$-$L_q$-regularity problem, we will simultaneously extend these results to the anisotropic mixed-norm setting.

\subsection*{Overview.} This paper is organized as follows.
In Section~\ref{sec:prelim} we discuss the required preliminaries for the rest of the paper. In particular, we introduce weighted mixed-norm anisotropic vector-valued function spaces of Bessel potential and Triebel-Lizorkin type. In Section~\ref{sec:pm;H} we extend the pointwise multiplier result \cite{Meyries&Veraar2015_pointwise_multiplication} to the setting of mixed-norm anisotropic Bessel potential spaces, and thereby provide a solution to the open problem posed on \cite[p.~342]{Amann2019_LQPP_II}.
In Section~\ref{sec:pm;F} we prove an analogous result for mixed-norm anisotropic Triebel-Lizorkin spaces.

\subsection*{Notation and convention.}
We write: $\N_0 = \{0,1,2,3,\ldots\}$, $\N = \{1,2,3,4,\ldots\}$, $\frac{1}{\N}=\{\frac{1}{n}: n \in \N\}$, $\R_{+}=(0,\infty)$,
$\hat{f}=\mathcal{F}f$ for the Fourier transform and $\check{f}=\mathcal{F}^{-1}f$ for the inverse Fourier transform.
Given a quasi-Banach space $Y$, we denote by $\mathcal{B}(Y)$ the space of bounded linear operators on $Y$.
Throughout the paper, we work over the field of complex scalars.

We use (modified) Vinogradov notation for estimates: $a \lesssim b$ means that there exists a constant $C \in (0,\infty)$ such that $a \leq C b$; $a \lesssim_{p,P} b$ means that there exists a constant $C \in (0,\infty)$, only depending on $p$ and $P$, such that $a \leq C b$; $a \eqsim b$ means $a \lesssim b$ and $b \lesssim a$; $a \eqsim_{p,P} b$ means $a \lesssim_{p,P} b$ and $b \lesssim_{p,P} a$.

\section{Preliminaries}\label{sec:prelim}

\subsection{Decompositions and Anisotropy}

Let $\ell \in \N$ and $\mathpzc{d} \in \N^\ell$. Consider the $\mathpzc{d}$-decomposition of $\R^d$:
\begin{equation}\label{eq:d_decomp}
\R^d= \R^{\mathpzc{d}_1} \times \ldots \times \R^{\mathpzc{d}_\ell}.
\end{equation}
We write $\R^\mathpzc{d}$ when we want to indicate that we view $\R^d$ as being $\mathpzc{d}$-decomposed as in \eqref{eq:d_decomp}. For $x \in \R^\mathpzc{d}$ we accordingly write $x=(x_1,\ldots,x_\ell)$ with $x_j \in \R^{\mathpzc{d}_j}$ for each $j \in \{1,\ldots,\ell\}$.

For $\mbs{a} \in (0,\infty)^\ell$ we define the anisotropic quasi-norm $\abs{\,\cdot\,}_{(\mbs{a},\mz{d})}\colon \R^{\mz{d}} \to \R_+$ by
\begin{equation*}
  \abs{x}_{(\mbs{a},\mz{d})} := \has{\sum_{j=1}^\ell\abs{x_j}^{2/a_j}}^{1/2}, \qquad x \in \R^{\mz{d}},
\end{equation*}
which is a quasi-metric on $\R^{\mz{d}}$.  
We also define the anisotropic dilation $\delta_\lambda^{(\mbs{a},\mz{d})}$ on $\R^{\mz{d}}$ for $\lambda>0$ by
\begin{equation*}
  \delta_\lambda^{(\mbs{a},\mz{d})}x:= (\lambda^{a_1}x_1,\cdots,\lambda^{a_\ell}x_\ell), \qquad x \in \R^{\mz{d}}.
\end{equation*}

\subsection{Weights}\label{PIBVP:subsec:sec:prelim;mixed-norm}

A reference for the general theory of Muckenhoupt weights is \cite[Chapter~9]{Grafakos_modern}.

A \emph{weight} on $\R^d$ is a measurable function $w:\R^d \longra [0,\infty]$ that takes it values almost everywhere in $(0,\infty)$.
We denote by $\mc{W}(\R^d)$ the set of all weights $w$ on $\R^d$.

For $w \in \mc{W}(\R^d)$ and $p \in [1,\infty)$ we denote by $L_{p}(\R^d,w)$ the space of all $f\in L_0(\R^d)$ with
\[
\norm{f}_{L^{p}(\R^d,w)} := \left( \int_{\R^d}|f(x)|^{p}w(x)\,\mathrm{d}x \right)^{1/p} < \infty,
\]
where $L_0(\R^d)$ the space of equivalence classes of complex-valued measurable functions on $\R^d$.

Suppose that $\R^d$ is $\mathpzc{d}$-decomposed as in \eqref{eq:d_decomp}.
For $\vec{p} \in [1,\infty)^{\ell}$ and $\vec{w} \in \prod_{j=1}^{l}\mathcal{W}(\R{\mz{d}_j})$ we denote by $L_{\vec{p}}(\R^{\mz{d}},\vec{w})$ the mixed-norm space
\[
L_{\vec{p}}(\R^{\mz{d}},\vec{w}) := L_{p_{l}}(\R^{\mz{d}_{\ell}},w_{\ell})[\ldots[L_{p_{1}}(\R^{\mz{d}_1},w_{1})]\ldots],
\]
that is, $L_{\vec{p}}(\R^{\mz{d}},\vec{w})$ is the space of all $f \in L_{0}(\R^d)$ with
\[
\norm{f}_{L_{\vec{p}}(\R^{\mz{d}},\vec{w})} :=
 \left( \int_{\R^{\mz{d}_{\ell}}} \ldots \left(\int_{\R^{\mz{d}_{1}}}|f(x)|^{p_{1}}w_{1}(x_{1})\mathrm{d}x_{1} \right)^{p_{2}/p_{1}} \ldots w_{\ell}(x_{\ell})\mathrm{d}x_{\ell} \right)^{1/p_{\ell}} < \infty.
\]
We equip $L_{\vec{p}}(\R^{\mz{d}},\vec{w})$ with the norm $\norm{\,\cdot\,}_{L_{\vec{p}}(\R^{\mz{d}},\vec{w})}$, which turns it into a Banach space.

Given a Banach space $X$, we denote by $L_{\vec{p}}(\R^{\mz{d}},\vec{w};X)$ the associated Bochner space
\[
L_{\vec{p}}(\R^{\mz{d}},\vec{w};X) := L_{\vec{p}}(\R^{\mz{d}},\vec{w})(X) = \{ f \in L_{0}(\R^{d};X) : \norm{f}_{X} \in L_{\vec{p}}(\R^{\mz{d}},\vec{w}) \}.
\]

For $p \in (1,\infty)$ we denote by $A_{p}=A_{p}(\R^{d})$ the class of all Muckenhoupt $A_{p}$-weights, which are all the locally integrable weights on $\R^d$ for which the $A_{p}$-characteristic $$[w]_{A_{p}}:=\sup_{Q\text{ cube in }\R^d} \left(\frac{1}{|Q|}\int_{Q} w(x)\,\mathrm{d}x\right)\left(\frac{1}{|Q|}\int_{Q} w(x)^{-\frac{1}{p-1}}\,\mathrm{d}x\right) \in [1,\infty]$$ is finite,
where in the supremum we consider cubes $Q$ in $\R^d$ with sides parallel to the coordinate axes.
We furthermore set $A_{\infty} := \bigcup_{p \in (1,\infty)}A_{p}$.

Let $p \in (1,\infty)$ and let $w$ be a weight on $\R^d$. Denoting by $w'_p = w^{-\frac{1}{p-1}}$ the $p$-dual weight of $w$, we have
\begin{equation*}\label{eq:p-dual_weight}
w \in A_p \quad \Longlra \quad w'_p \in A_{p'} \quad \Longlra \quad w,w'_p \in A_\infty.    
\end{equation*}

The most important weight for this paper is the power weight
\begin{equation*}
w_\gamma (x_1,x') = |x_1|^\gamma, \qquad x=(x_1,x') \in \R \times \R^{d-1},
\end{equation*}
where $\gamma \in \R$.
For this weight we have
\begin{equation*}
w_\gamma  \in A_p  \quad \Longlra \quad \gamma \in (-1,p-1) 
\end{equation*}
and
\begin{equation*}
w_\gamma  \in A_\infty  \quad \Longlra \quad \gamma \in (-1,\infty). 
\end{equation*}
Furthermore, $(w_\gamma)'_p = w_{\gamma'_p}$ with $\gamma'_p = -\frac{\gamma}{p-1}$.

\subsection{UMD Spaces}\label{subsec:UMD}
We refer the reader to \cite[Chapter~4]{Hytonen&Neerven&Veraar&Weis2016_Analyis_in_Banach_Spaces_I} for an introduction to the theory of UMD spaces.

Let us state some facts:
\begin{itemize}
\item Every Hilbert space is a UMD space;
\item If $X$ is a UMD space, $(S, \Sigma, \mu)$ is $\sigma$-finite and $p\in (1, \infty)$, then $L_p(S;X)$ is a UMD space.
A generalization of this to the setting of reflexive Musielak-Orlicz spaces can be found in \cite{LVY2018}.
\item UMD spaces are reflexive. In particular, $L^1$ and $L^\infty$ are not UMD.
\item Closed subspaces and quotients of UMD spaces are again UMD spaces. In particular, reflexive Besov, Triebel-Lizorkin and Sobolev spaces are UMD.
\end{itemize}

\subsection{Function Spaces}

For the theory of Banach space-valued distributions we refer the reader to
\cite{Amann2019_LQPP_II}. Let us explicitely mention that, given a Banach space $X$, we denote by $\mathcal{S}(\R^d;X)$ the space of $X$-valued Schwartz functions on $\R^d$ and we denote by $\mathcal{S}'(\R^d;X) = \mathcal{L}(\mathcal{S}(\R^d),X)$ the space of $X$-valued tempered distributions on $\R^d$. 
Furthermore, we denote by $\mathscr{O}_{\mathrm{M}}(\R^d;X)$ the space of slowly increasing smooth functions on $\R^d$. This means that $f \in \mathscr{O}_{\mathrm{M}}(\R^d;X)$ if and only if $f \in C^\infty(\R^d;X)$ and, for each $\alpha \in \N^d$, there exist $m_\alpha \in \N$ and $c_\alpha >0 $ such that
$$
\norm{D^\alpha f(x)}_{X} \leq c_\alpha (1+|x|^2)^{m_\alpha}, \qquad x \in \R^d. 
$$

Let $X$ be a Banach space, let $\ell \in \N$ and let $\mathpzc{d} \in \N^\ell$. 
For $\vec{a} \in (\frac{1}{\N})^\ell$ and $s \in \R$ we define the anisotropic Bessel potential operator $J^{(\vec{a},\mathpzc{d})}_{s} \in \mc{L}(\mc{S}'(\R^d;X))$ by
\begin{equation*}
  J^{(\vec{a},\mathpzc{d})}_{s}f := \left[ \hab{(1+\abs{\xi}_{(\mbs{a},\mz{d})}^{2})^{s/2}\cdot \widehat{f}}\right]^{\vee}, \qquad\qquad f \in \mc{S}'(\R^d;X),
\end{equation*}
and for $\vec{p} \in (1,\infty)^\ell$ and $\vec{w} \in \prod_{j=1}^{\ell}A_{p_j}(\R^{\mz{d}_j})$ we define the corresponding Bessel potential space $H^{s,\vec{a}}_{\vec{p}}(\R^{\mathpzc{d}},\vec{w};X)$ by
\begin{equation*}
\begin{aligned}
&H^{s,\vec{a}}_{\vec{p}}(\R^{\mathpzc{d}},\vec{w};X) := \left\{ f \in \mc{S}'(\R^d;X) : J^{(\vec{a},\mathpzc{d})}_{s}f \in L_{\vec{p}}(\R^{\mathpzc{d}},\vec{w};X) \right\}, \\
&\norm{f}_{H^{s,\vec{a}}_{\vec{p}}(\R^{\mathpzc{d}},\vec{w};X)} := \norm{J^{(\vec{a},\mathpzc{d})}_{s}f}_{L_{\vec{p}}(\R^{\mz{p}},\vec{w};X)}.
\end{aligned}
\end{equation*}
Furthermore, for $\sigma \in \R$, $\vec{\omega} \in \N^\ell$, $\vec{p} \in (1,\infty)^\ell$ and $\vec{w} \in \prod_{j=1}^{\ell}A_{p_j}(\R^{\mz{d}_j})$ we define the Bessel potential space $H^{\sigma/\vec{\omega}}_{\vec{p}}(\R^{\mathpzc{d}},\vec{w};X)$ by
$$
H^{\sigma/\vec{\omega}}_{\vec{p}}(\R^{\mathpzc{d}},\vec{w};X) := H^{\sigma/\dot{\omega},\vec{\omega}/\dot{\omega}}_{\vec{p}}(\R^{\mathpzc{d}},\vec{w};X),
$$
where $\dot{\omega} = \mathrm{lcm}(\vec{\omega})$, the least common multiple of $\omega_1,\ldots,\omega_\ell$. 
The Bessel potential space $H^{\sigma/\vec{\omega}}_{\vec{p}}(\R^{\mathpzc{d}},\vec{w};X)$ is a weighted version of the Bessel potential space $H^{\sigma/\vec{\omega}}_{\vec{p}}(\R^{\mathpzc{d}};X)$ from \cite{Amann2019_LQPP_II}, also see \cite{Koehne&Saal2020}.

Let $\vec{a} \in (\frac{1}{\N})^\ell$ and $s \in \R$. Suppose that $\mathrm{gcd}(\vec{a}^{-1})=1$, the greatest common divisor of $\frac{1}{a_1},\ldots,\frac{1}{a_\ell}$, and write $m=(\frac{1}{a_1} \cdot \ldots \cdot \frac{1}{a_\ell})$.
Then, setting $\vec{\omega}=m\vec{a} \in \N^\ell$ and $\sigma=ms$, we have $m=\dot{\omega}=\mathrm{lcm}(\vec{\omega})$, so that $s=\sigma/\dot{\omega}$ and $\vec{a}=\vec{\omega}/\dot{\omega}$, and thus
\begin{equation}\label{eq:prelim:H_coincidence_Amann_def}
H^{s,\vec{a}}_{\vec{p}}(\R^{\mathpzc{d}},\vec{w};X) = H^{\sigma/\vec{\omega}}_{\vec{p}}(\R^{\mathpzc{d}},\vec{w};X), \qquad\qquad \text{isometrically}.    
\end{equation}
The following proposition makes it possible to obtain a version of this for the case  $\mathrm{gcd}(\vec{a}^{-1}) > 1$, see Corollary~\ref{cor:lem:H_scaling_s-a}.

\begin{proposition}\label{prop:H_scaling_s-a}
Let $X$ be a UMD Banach space, $\vec{a} \in (\frac{1}{\N})^\ell$, $\vec{p} \in (1,\infty)^\ell$, $\vec{w} \in \prod_{j=1}^\ell A_{p_j}(\R^{\mathpzc{d}_j})$ and $s \in \R$. Let $k \in \N$ be such that $k\vec{a} \in (\frac{1}{\N})^\ell$.  
Then
\begin{equation}\label{eq:lem:H_scaling_s-a}
H^{s,\vec{a}}_{\vec{p}}(\R^{\mathpzc{d}},\vec{w};X) = H^{k s,k \vec{a}}_{\vec{p}}(\R^{\mathpzc{d}},\vec{w};X).    
\end{equation}
\end{proposition}
\begin{proof}
We need to show that $J^{(\vec{a},\mathpzc{d})}_{s}J^{(k\vec{a},\mathpzc{d})}_{-k s}$ and $J^{(k\vec{a},\mathpzc{d})}_{k s}J^{(\vec{a},\mathpzc{d})}_{-s}$ are bounded operators on $L_{\vec{p}}(\R^{\mathpzc{d}},\vec{w};X)$.
By relabeling the parameters, it suffices to consider the first. So we need to show that
$$
m(\xi) := \frac{(1+\abs{\xi}_{(\mbs{a},\mz{d})}^{2})^{s/2}}{(1+\abs{\xi}_{(k\mbs{a},\mz{d})}^{2})^{k s/2}}, \qquad\qquad \xi \in \R^{\mz{d}},
$$
defines a Fourier multiplier on $L_{\vec{p}}(\R^{\mathpzc{d}},\vec{w};X)$.
By by using the weighted anisotropic mixed-norm Mikhlin multiplier theorem \cite[Theorem~7.1]{Lorist2019_pointwise}, this can be done through a standard scaling argument, as follows.

We define the function $M:\R^{d+1}\setminus \{0\} \to \R$ by
$$
M(t,\xi) := \frac{(t^{2k}+\abs{\xi}_{(\mbs{a},\mz{d})}^{2})^{s/2}}{(t^2+\abs{\xi}_{(k\mbs{a},\mz{d})}^{2})^{k s/2}}, \qquad\qquad (t,\xi) \in \R \times \R^{\mz{d}} \setminus \{(0,0)\},
$$
which is a $C^\infty$-function as  $\vec{a},k\vec{a} \in (\frac{1}{\N})^\ell$. Furthermore, note that $M$ is homogeneous of order $0$ with respect to the scaling $(t,\xi) \mapsto (\lambda t,\delta^{(\vec{a},\mz{d})}_\lambda \xi)$, that is, $M(t,\xi) = M(\lambda t,\delta^{(k\vec{a},\mz{d})}_\lambda \xi)$ for all $\lambda > 0$. This homogeneity implies that
$$
D^{(0,\beta)}M(t,\xi) = \lambda^{k\vec{a}\cdot\beta}[D^{(0,\beta)}M](\lambda t,\delta^{(k\vec{a},\mz{d})}_\lambda \xi), \qquad \beta \in \N_0^{\mz{d}}, \lambda > 0, (t,\xi) \in \R \times \R^{\mz{d}} \setminus \{(0,0)\},
$$
where $k\vec{a}\cdot\beta=k\sum_{j=1}^{\ell}a_j\beta_j$. Therefore, for each $\beta \in \N_0^{\mz{d}}$,
$$
M_{\beta}(t,\xi) := \abs{\xi}_{(k\mbs{a},\mz{d})}^{k\vec{a}\cdot\beta}D^{(0,\beta)}M(t,\xi), \qquad\qquad (t,\xi) \in \R \times \R^{\mz{d}} \setminus \{(0,0)\},
$$
defines a $C^\infty$-function that is homogeneous of order $0$ with respect to the scaling $(t,\xi) \mapsto (\lambda t,\delta^{(\vec{a},\mz{d})}_\lambda \xi)$ and thus is a bounded function.
In particular, taking $t=1$, we find that
$$
\sup_{\xi \in \R^{\mz{d}}}\abs{\xi}_{(k\mbs{a},\mz{d})}^{k\vec{a}\cdot\beta}\abs{D^{\beta}m(\xi)} < \infty, \qquad \beta \in \N_0^{\mz{d}}.
$$
We can thus apply the weighted anisotropic mixed-norm Mikhlin multiplier theorem \cite[Theorem~7.1]{Lorist2019_pointwise} to conclude that $m$ is Fourier multiplier on $L_{\vec{p}}(\R^{\mathpzc{d}},\vec{w};X)$.
\end{proof}

\begin{corollary}\label{cor:lem:H_scaling_s-a}
Let $X$ be a UMD Banach space, $\vec{a} \in (\frac{1}{\N})^\ell$, $\vec{p} \in (1,\infty)^\ell$, $\vec{w} \in \prod_{j=1}^\ell A_{p_j}(\R^{\mathpzc{d}_j})$ and $s \in \R$.
Setting $m=(\frac{1}{a_1} \cdot \ldots \cdot \frac{1}{a_\ell})\,[\mathrm{gcd}(\vec{a}^{-1})]^{1-\ell}$, $\vec{\omega}=m\vec{a}$ and $\sigma=ms$, we have 
\begin{equation}\label{eq:cor:lem:H_scaling_s-a}
H^{s,\vec{a}}_{\vec{p}}(\R^{\mathpzc{d}},\vec{w};X) = H^{\sigma/\vec{\omega}}_{\vec{p}}(\R^{\mathpzc{d}},\vec{w};X). 
\end{equation}
\end{corollary}
\begin{proof}
Set $k=\mathrm{gcd}(\vec{a}^{-1})$.
By Proposition~\ref{prop:H_scaling_s-a}, there is the identity \eqref{eq:lem:H_scaling_s-a}. As $\mathrm{gcd}((k\vec{a})^{-1})=1$, $\vec{\omega}=\tilde{m}(k\vec{a})$ and $\sigma=\tilde{m}(ks)$ with $\tilde{m}=(\frac{1}{a_1} \cdot \ldots \cdot \frac{1}{a_\ell})k^{-\ell}$, we can apply \eqref{eq:prelim:H_coincidence_Amann_def} to the right-hand side of \eqref{eq:lem:H_scaling_s-a} to obtain that
$$
H^{s,\vec{a}}_{\vec{p}}(\R^{\mathpzc{d}},\vec{w};X) = H^{k s,k \vec{a}}_{\vec{p}}(\R^{\mathpzc{d}},\vec{w};X) = H^{\sigma/\vec{\omega}}_{\vec{p}}(\R^{\mathpzc{d}},\vec{w};X).
$$
\end{proof}

In light of the above corollary, the following intersection representation is an extension of \cite[Theorem~VII.4.6.1]{Amann2019_LQPP_II}.
\begin{theorem}\label{thm:IR;H}
Let $X$ be a UMD Banach space, $\vec{a} \in (\frac{1}{\N})^\ell$, $\vec{p} \in (1,\infty)^\ell$, $\vec{w} \in \prod_{j=1}^\ell A_{p_j}(\R^{\mathpzc{d}_j})$ and $s \in (0,\infty)$.
For each $j \in \{1,\ldots,\ell\}$, let 
\begin{equation}\label{eq:thm:IR;H;notation}
\begin{aligned}
\check{\mathpzc{d}}_{j}=(\mathpzc{d}_{j+1},\ldots,\mathpzc{d}_{\ell}), & \check{p}_j=(p_{j+1},\ldots,p_\ell), & \check{\vec{w}}_{j} = (w_{j+1},\ldots,w_\ell), \\
\hat{\mathpzc{d}}_{j}=(\mathpzc{d}_{1},\ldots,\mathpzc{d}_{j-1}), & \hat{p}_j=(p_{1},\ldots,p_{j-1}), & \hat{\vec{w}}_{j} = (w_{1},\ldots,w_{j-1}).
\end{aligned}    
\end{equation}
Then
\begin{equation}\label{eq:thm:IR;H}
H^{s,\vec{a}}_{\vec{p}}(\R^{\mathpzc{d}},\vec{w};X) = \bigcap_{j=1}^{\ell} L_{\check{\vec{p}}_{j}}\Big(\R^{\check{\mathpzc{d}}_{j}},\check{\vec{w}}_{j};H^{s/a_j}_{p_j}\big(\R^{\mathpzc{d}_j},w_j;L_{\hat{\vec{p}}_{j}}(\R^{\hat{\mathpzc{d}}_{j}},\hat{\vec{w}}_{j};X)\big)\Big). 
\end{equation}
\end{theorem}
\begin{proof}
Let the notation be as in Corollary~\ref{cor:lem:H_scaling_s-a}.
Then we have the identity \eqref{eq:cor:lem:H_scaling_s-a} for the space on the left-hand side of \eqref{eq:thm:IR;H}.
The case that $X$ is a so-called $\vec{\omega}$-admissible Banach space, $\vec{p}=p\vec{1}$ and $\vec{w}=\vec{1}$ is thus contained in \cite[Theorem~VII.4.6.1]{Amann2019_LQPP_II}, where the Banach space $X$ is said to be $\vec{\omega}$-admissible if it is a UMD space which additionally
has Pisier's property $\mathrm{(\alpha)}$ if $\vec{\omega} \neq \dot{\omega}\vec{1}$. Let us comment on how the proof given in \cite[Theorem~VII.4.6.1]{Amann2019_LQPP_II} remains valid for the case we consider here.

As discussed in \cite[Section~4.7]{Amann2019_LQPP_II}, $\vec{\omega}$-admissibility of the Banach space $X$ guarantees the validity of the Mikhlin, respectively Marcinkiewicz, Fourier multiplier theorem on $L_{p}(\R^d;X)$: Mikhlin' theorem holds when $X$ is a UMD space and Marcinkiewicz' theorem holds when $X$ is an $\mathrm{(\alpha)}$-UMD space. 
However, the additional assumption of Pisier's property $\mathrm{(\alpha)}$ can be avoided by using the weighted anisotropic mixed-norm Mikhlin multiplier theorem \cite[Theorem~7.1]{Lorist2019_pointwise}, but also see
\cite[Theorem~3.2]{Hytonen_anisotropic} 
and \cite[Theorem~5.1.4]{Fackler&Hytonen&Lindemulder2018} for the cases $\vec{w}=1$ and $\vec{w} \in \prod_{j=1}^\ell A_{p_j}^{\mathrm{rec}}(\R^{\mathpzc{d}_j})$, respectively.
\end{proof}

As in \cite[Theorem~3.7.2]{Amann09}, the above intersection representation dualizes to the following sum representation.

\begin{cor}\label{cor:thm:IR;H;SR}
Let $X$ be a UMD Banach space, $\vec{a} \in (0,\infty)^\ell$, $\vec{p} \in (1,\infty)^\ell$, $\vec{w} \in \prod_{j=1}^\ell A_{p_j}(\R^{\mathpzc{d}_j})$ and $s \in (-\infty,0)$.
Then
\begin{equation*}
H^{s,\vec{a}}_{\vec{p}}(\R^{\mathpzc{d}},\vec{w};X) = \foo_{j=1}^{\ell} L_{\check{\vec{p}}_{j}}\Big(\R^{\check{\mathpzc{d}}_{j}},\check{\vec{w}}_{j};H^{s/a_j}_{p_j}\big(\R^{\mathpzc{d}_j},w_j;L_{\hat{\vec{p}}_{j}}(\R^{\hat{\mathpzc{d}}_{j}},\hat{\vec{w}}_{j};X)\big)\Big), 
\end{equation*}
where we use the notation from \eqref{eq:thm:IR;H;notation}.
\end{cor}

Next we introduce the class of anisotropic Littlewood--Paley sequences. We define  $\Phi^{(\vec{a},\mathpzc{d})}(\R^d)$ as the set of all sequences $(\varphi_n)_{n\in \N_0}\subseteq \mc{S}(\R^d)$ constructed as follows: given $\varphi_0 \in \mc{S}(\R^d)$ satisfying
\begin{align*}
\begin{cases}
   0\leq \widehat{\varphi}_0 \leq 1,\\
\widehat{\varphi}_0(\xi)=1, & \abs{\xi}_{(\mbs{a},\mz{d})} \leq 1,\\
\widehat{\varphi}_0(\xi)=0, & \abs{\xi}_{(\mbs{a},\mz{d})} \geq 2, 
\end{cases}
\end{align*}
we define $\varphi_n\in\mc{S}(\R^d)$ for $n \geq 1$ by
\begin{equation*}
  \widehat{\varphi}_n(\xi) := \widehat{\varphi}_0\hab{\delta^{(\mbs{a},\mz{d})}_{2^{-n}} \xi} - \widehat{\varphi}_0\hab{\delta^{(\mbs{a},\mz{d})}_{2^{-(n-1)}} \xi}, \qquad \xi \in \R^d.
\end{equation*}
Note that for $(\varphi_n)_{n\in \N_0} \in \Phi^{(\vec{a},\mathpzc{d})}(\R^d)$ we have $\sum_{n=0}^\infty \widehat{\varphi}_n =1$ with
\begin{align*}
\supp(\widehat{\varphi}_0) &\subseteq \{ \xi \in \R^d : |\xi|_{(\vec{a},\mathpzc{d})} \leq 2\}, \\
\supp(\widehat{\varphi}_n) &\subseteq \{ \xi \in \R^d : 2^{n-1}\leq |\xi|_{(\vec{a},\mathpzc{d})} \leq 2^{n+1}\}, &&n \geq 1.
\end{align*}

To $\varphi \in \Phi^{\mathpzc{d},a}(\R^{n})$ we associate the family of convolution operators
$(S_{n})_{n \in \N_0} = (S_{n}^{\varphi})_{n \in \N_0}   \subset \mathcal{L}(\mathcal{S}'(\R^{d};X),\mathscr{O}_{M}(\R^{d};X)) \subset \mathcal{L}(\mathcal{S}'(\R^{d};X))$ given by
\begin{equation*}\label{functieruimten:eq:convolutie_operatoren}
S_{n}f = S_{n}^{\varphi}f := \varphi_{n} * f = [\hat{\varphi}_{n}\hat{f}]^{\vee}, \qquad\qquad f \in \mathcal{S}'(\R^{d};X).
\end{equation*}
It holds that $f = \sum_{n=0}^{\infty}S_{n}f$ in $\mathcal{S}'(\R^{d};X)$ respectively in $\mathcal{S}(\R^{d};X)$ whenever $f \in \mathcal{S}'(\R^{d};X)$ respectively $f \in \mathcal{S}(\R^{d};X)$.

Let $X$ be a Banach space, $\vec{a} \in (0,\infty)^{\ell}$, $\vec{p} \in [1,\infty)^{
\ell}$, $q \in [1,\infty]$, $s \in \R$, and $\vec{w} \in \prod_{j=1}^{\ell}A_{\infty}(\R^{\mathpzc{d}_{j}})$. 
We define the Triebel-Lizorkin space $F^{s,\vec{a}}_{\vec{p},q}(\R^{\mathpzc{d}},\vec{w};X)$ as the Banach space of all $f \in \mathcal{S}'(\R^{d};X)$ for which
\[
\norm{f}_{F^{s,\vec{a}}_{\vec{p},q}(\R^{\mathpzc{d}},\vec{w};X)}
:= \norm{(2^{ns}S_{n}^{\varphi}f)_{n \in \N_0}}_{L_{\vec{p}}(\R^{\mathpzc{d}},\vec{w})[\ell_{q}(\N_0)](X)} < \infty.
\]	
Up to an equivalence of extended norms on $\mathcal{S}'(\R^{d};X)$, $\norm{\,\cdot\,}_{F^{s,\vec{a}}_{\vec{p},q}(\R^{\mathpzc{d}},\vec{w};X)}$ does not depend on the particular choice of $\varphi \in \Phi^{\mathpzc{d},\vec{a}}(\R^{d})$.

Let $X$ be a Banach space, $\vec{a} \in (0,\infty)^{\ell}$, $\vec{p} \in [1,\infty)^{
\ell}$, $q \in [1,\infty]$, $s \in \R$, $\vec{w} \in \prod_{j=1}^{\ell}A_{\infty}(\R^{\mathpzc{d}_{j}})$ and let $E$ be a 
quasi-Banach function space on a $\sigma$-finite measure space $(S,\mathscr{A},\mu)$ with the property that $E^{[r]}$ is a UMD Banach function space for some $r \in (0,\infty)$, where
$$
E^{[r]} = \{ f \in L_0(S) : |f|^{1/r} \in E \}, \qquad \norm{f}_{E^[r]} = \norm{\,|f|^{1/r}\,}_{E}^r. 
$$
We define the generalized Triebel-Lizorkin space $\F^{s}_{p,q}(\R^{\mz{d}},\vec{w};E;X)$ (see \cite[Example~3.20]{Lindemulder2019_IR}) as the Banach space of all $f \in L_0(S;\mathcal{S}'(\R^{d};X))$ for which
\[
\norm{f}_{\F^{s,\vec{a}}_{\vec{p},q}(\R^{\mathpzc{d}},\vec{w};E;X)}
:= \norm{(2^{ns}S_{n}^{\varphi}f)_{n \in \N_0}}_{L_{\vec{p}}(\R^{\mathpzc{d}},\vec{w})[E[\ell_{q}(\N_0)]](X)} < \infty.
\]	
Up to an equivalence of extended norms on $L_0(S;\mathcal{S}'(\R^{d};X))$, $\norm{\,\cdot\,}_{\F^{s,\vec{a}}_{\vec{p},q}(\R^{\mathpzc{d}},\vec{w};E;X)}$ does not depend on the particular choice of $\varphi \in \Phi^{\mathpzc{d},\vec{a}}(\R^{d})$.
We refer the reader to  \cite[Section~2.3]{Lindemulder2019_IR} for the definition of $L_0(S;\mathcal{S}'(\R^{d};X))$.

\begin{proposition}\label{prop:prelim:duality}
Let $X$ be a Banach space, $\vec{a} \in (0,\infty)^\ell$, $\vec{p} \in (1,\infty)^\ell$, $q \in [1,\infty)$, $\vec{w} \in \prod_{j=1}^\ell A_{p_j}(\R^{\mathpzc{d}_j})$ and $s \in \R$. 
Denote by $\vec{p}'$ the H\"older conjugate vector of $\vec{p}$, by $q'$ the H\"older conjugate of $q$ and by $\vec{w}'_{\vec{p}}=(w_1^{-\frac{1}{p_1-1}},\ldots,w_\ell^{-\frac{1}{p_\ell-1}})$ the $\vec{p}$-dual weight vector of $\vec{w}$.
Then 
\begin{equation}\label{eq:cor:thm:pm_F;dual;q-finite}
[F^{s,\vec{a}}_{\vec{p},q}(\R^{\mathpzc{d}},\vec{w};X)]^{*} =
F^{-s,\vec{a}}_{\vec{p}',q'}(\R^{\mathpzc{d}},\vec{w}'_{\vec{p}};X^{*})    
\end{equation}
under the natural pairing (induced by $\mathcal{S}'(\R^d;X^*) = [\mathcal{S}(\R^d;X)]'$).
Moreover,
\begin{equation}\label{eq:cor:thm:pm_F;dual;q-finite;norming}
\norm{f}_{F^{s,\vec{a}}_{\vec{p},q}(\R^{\mathpzc{d}},\vec{w};X)} \lesssim \sup\left\{\ip{f}{g} : g \in \mathcal{S}(\R^d;X^*), \norm{g}_{F^{-s,\vec{a}}_{\vec{p}',q'}(\R^{\mathpzc{d}},\vec{w}'_{\vec{p}};X^{*})} \leq 1\right\} =: \normm{f}_{F^{s,\vec{a}}_{\vec{p},q}(\R^{\mathpzc{d}},\vec{w};X)}  
\end{equation}
for all $f \in F^{s,\vec{a}}_{\vec{p},q}(\R^{\mathpzc{d}},\vec{w};X)$, where
\begin{equation}\label{eq:cor:thm:pm_F;dual;q-finite;norming;pairing}
\ip{f}{g} = \ip{f}{g}_{\ip{F^{s,\vec{a}}_{\vec{p},q}(\R^{\mathpzc{d}},\vec{w};X)}{F^{-s,\vec{a}}_{\vec{p}',q'}(\R^{\mathpzc{d}},\vec{w}'_{\vec{p}};X^{*})}}
= \ip{f}{g}_{\ip{\mathcal{S}'(\R^d;X)}{\mathcal{S}(\R^d;X^*)}}.
\end{equation}
\end{proposition}
\begin{proof}
The duality statement \eqref{eq:cor:thm:pm_F;dual;q-finite;norming} is contained in \cite[Example~6.4]{Lindemulder2019_IR}. Furthermore, note that the second identity in \eqref{eq:cor:thm:pm_F;dual;q-finite;norming;pairing} follows from $\mathcal{S}(\R^d;X) \stackrel{d}{\hookrightarrow} F^{s,\vec{a}}_{\vec{p},q}(\R^{\mathpzc{d}},\vec{w};X) \hookrightarrow \mathcal{S}'(\R^d;X)$.

In order to prove \eqref{eq:cor:thm:pm_F;dual;q-finite;norming}, let $f \in F^{s,\vec{a}}_{\vec{p},q}(\R^{\mathpzc{d}},\vec{w};X)$. 
By \eqref{eq:cor:thm:pm_F;dual;q-finite} there exists $h \in F^{-s,\vec{a}}_{\vec{p}',q'}(\R^{\mathpzc{d}},\vec{w}'_{\vec{p}};X^{*})$ with norm $\norm{h}_{F^{-s,\vec{a}}_{\vec{p}',q'}(\R^{\mathpzc{d}},\vec{w}'_{\vec{p}};X^{*})} \leq 1$ such that $\norm{f}_{F^{s,\vec{a}}_{\vec{p},q}(\R^{\mathpzc{d}},\vec{w};X)} \leq 2|\ip{f}{h}|$.
By density of $\mathcal{S}(\R^d;X)$ in $F^{s,\vec{a}}_{\vec{p},q}(\R^{\mathpzc{d}},\vec{w};X)$ there exists a sequence $(f_k)_{k \in \N} \subset \mathcal{S}(\R^d;X)$ such that $f=\lim_{k \to \infty}f_k$ in $F^{s,\vec{a}}_{\vec{p},q}(\R^{\mathpzc{d}},\vec{w};X)$. 
From \cite[Lemma~3.21]{Lindemulder2019_IR} applied to $h=\sum_{m=0}^{\infty}S_m h$ in $\mathcal{S}'(\R^d;X)$ with estimates corresponding to $F^{-s,\vec{a}}_{\vec{p}',q'}(\R^{\mathpzc{d}},\vec{w}'_{\vec{p}};X^{*})$ and an approximation argument as in \cite[Lemma~3.8]{Meyries&Veraar2012_sharp_embedding} it follows that there exists a sequence $(g_n)_{n \in \N} \subset \mathcal{S}(\R^d;X^*)$ with $\sup_{n \in \N}\norm{g_n}_{F^{-s,\vec{a}}_{\vec{p}',q'}(\R^{\mathpzc{d}},\vec{w}'_{\vec{p}};X^{*})} \lesssim \norm{h}_{F^{-s,\vec{a}}_{\vec{p}',q'}(\R^{\mathpzc{d}},\vec{w}'_{\vec{p}};X^{*})} \leq 1$ and $h=\lim_{n \to \infty}g_n$ in $\mathcal{S}'(\R^d;X^*)$.
Now note that
$$
\ip{f}{h} = \lim_{k \to \infty}\ip{f_k}{h}=\lim_{k \to \infty}\lim_{n \to \infty}\ip{f_k}{g_n},
$$
so that
$$
\norm{f}_{F^{s,\vec{a}}_{\vec{p},q}(\R^{\mathpzc{d}},\vec{w};X)} \leq 2|\ip{f}{h}| \lesssim \limsup_{k \to \infty}\normm{f_k}_{F^{s,\vec{a}}_{\vec{p},q}(\R^{\mathpzc{d}},\vec{w};X)} 
\leq \normm{f}_{F^{s,\vec{a}}_{\vec{p},q}(\R^{\mathpzc{d}},\vec{w};X)} + \limsup_{k \to \infty}\normm{f_k-f}_{F^{s,\vec{a}}_{\vec{p},q}(\R^{\mathpzc{d}},\vec{w};X)}.
$$
As $\normm{\,\cdot\,}_{F^{s,\vec{a}}_{\vec{p},q}(\R^{\mathpzc{d}},\vec{w};X)} \leq \norm{\,\cdot\,}_{F^{s,\vec{a}}_{\vec{p},q}(\R^{\mathpzc{d}},\vec{w};X)}$ by \eqref{eq:cor:thm:pm_F;dual;q-finite}, we have $\limsup_{k \to \infty}\normm{f_k-f}_{F^{s,\vec{a}}_{\vec{p},q}(\R^{\mathpzc{d}},\vec{w};X)} = 0$.
This shows that \eqref{eq:cor:thm:pm_F;dual;q-finite;norming} holds true.
\end{proof}

\begin{proposition}\label{prop:complex_interpolation}
Let $X$ be a Banach space, $\vec{a} \in (0,\infty)^{\ell}$, $\vec{p} \in [1,\infty)^\ell$, $q \in [1,\infty]$ and $\vec{w} \in \prod_{j=1}^\ell A_\infty(\R^{\mathpzc{d}_j})$. Let $s, s_0, s_1 \in \R$ and $\theta \in [0,1]$ be such that $s=s_0(1-\theta)+s_1\theta$. Then
\begin{equation}\label{eq:prop:complex_interpolation}
[F^{s_0,\vec{a}}_{\vec{p},q}(\R^{\mathpzc{d}},\vec{w};X),F^{s_1,\vec{a}}_{\vec{p},q}(\R^{\mathpzc{d}},\vec{w};X)]_{\theta} = F^{s,\vec{a}}_{\vec{p},q}(\R^{\mathpzc{d}},\vec{w};X).    
\end{equation}
\end{proposition}
\begin{proof}
We may without loss of generality assume that $s_0 < s_1$. 
For each $\sigma \in \R$ we define 
$$
m_\sigma(\xi) := \left(\sum_{j=1}^{\ell}(1+|\xi_j|^2)^{1/a_j}\right)^{\sigma/2}, \qquad\qquad \xi \in \R^{d},
$$
and we define $B_{\sigma} \in \mc{L}(\mc{S}'(\R^d;X))$ by
\begin{equation*}
  B_{\sigma}f := \left[m_{\sigma}\widehat{f}\right]^{\vee}, \qquad\qquad f \in \mc{S}'(\R^d;X).
\end{equation*}
As in \cite[Proposition~3.25]{Lindemulder2019_IR} (or the reference \cite[Proposition~5.2.28]{Lindemulder_master-thesis} given there), it can be shown that
$$
B_{\sigma}:F^{s_0+\sigma,\vec{a}}_{\vec{p},q}(\R^{\mathpzc{d}},\vec{w};X) \to F^{s_0,\vec{a}}_{\vec{p},q}(\R^{\mathpzc{d}},\vec{w};X), \qquad \text{as an isomorphism of Banach spaces}.
$$
Since $F^{s_0,\vec{a}}_{\vec{p},q}(\R^{\mathpzc{d}},\vec{w};X)$ is a $(\mz{d},\vec{a})$-admissible Banach space of tempered distributions in the sense of \cite[pg.~66]{Lindemulder2020_JEE} by \cite[Proposition~5.2.26]{Lindemulder_master-thesis}, the complex interpolation identity \eqref{eq:prop:complex_interpolation} can be shown as in \cite[Theorem~2.3.2]{Amann09}.
\end{proof}

\begin{lemma}\label{lemma:incl_into_L1_loc}
Let $X$ be a Banach space, $E$ a quasi-Banach function space on a $\sigma$-finite measure space $(S,\ms{A},\mu)$, $\vec{a} \in (0,\infty)^\ell$, $\vec{p} \in (0,\infty)^\ell$, $q \in (0,\infty]$, $\vec{w} \in \prod_{j=1}^\ell A_\infty(\R^{\mathpzc{d}_j})$ and $s \in (0,\infty)$. Let $\vec{r} \in (0,\infty)^{\ell}$ be such that $r_{j} < p_{1} \wedge \ldots \wedge p_{j} \wedge q$ for $j=1,\ldots,\ell$, $\vec{w} \in \prod_{j=1}^{\ell}A_{p_{j}/r_{j}}(\R^{\mathpzc{d}_{j}})$ and $E^{[\vec{r}_{\max}]}$ is a UMD Banach function space, where $\vec{r}_{\max} = \max\{r_1,\ldots,r_\ell\}$.
Assume that $s>\sum_{j=1}^{\ell}a_j\mathpzc{d}_j(\frac{1}{r_j}-1)_+$. 
Then 
\begin{equation}\label{eq:lemma:incl_into_L1_loc}
\F^{s,\vec{a}}_{\vec{p},q}(\R^{\mathpzc{d}},\vec{w};E;X) \hookrightarrow L_0(S;L_{1,\loc}(\R^{\mz{d}};X)).   
\end{equation}
\end{lemma}
\begin{proof}
The inclusion \eqref{eq:lemma:incl_into_L1_loc} follows from \cite[Theorem~3.22]{Lindemulder2019_IR} (and \cite[Examples~3.5$\&$3.20]{Lindemulder2019_IR}). 
\end{proof}

\begin{theorem}[\cite{Lindemulder2019_IR}]\label{thm:IR}
Let $X$ be a Banach space, $\vec{a} \in (0,\infty)^\ell$, $\vec{p} \in (0,\infty)^\ell$, $q \in (0,\infty]$, $\vec{w} \in \prod_{j=1}^\ell A_\infty(\R^{\mathpzc{d}_j})$ and $s \in (0,\infty)$. Let $\vec{r} \in (0,\infty)^{\ell}$ be such that $r_{j} < p_{1} \wedge \ldots \wedge p_{j} \wedge q$ for $j=1,\ldots,\ell$ and $\vec{w} \in \prod_{j=1}^{\ell}A_{p_{j}/r_{j}}(\R^{\mathpzc{d}_{j}})$.
For each $j \in \{1,\ldots,\ell\}$, let 
$$
\begin{array}{lll}
\check{\mathpzc{d}}_{j}=(\mathpzc{d}_{j+1},\ldots,\mathpzc{d}_{\ell}), & \check{p}_j=(p_{j+1},\ldots,p_\ell), & \check{\vec{w}}_{j} = (w_{j+1},\ldots,w_\ell), \\
\hat{\mathpzc{d}}_{j}=(\mathpzc{d}_{1},\ldots,\mathpzc{d}_{j-1}), & \hat{p}_j=(p_{1},\ldots,p_{j-1}), & \hat{\vec{w}}_{j} = (w_{1},\ldots,w_{j-1}).
\end{array}
$$
If $s>\sum_{j=1}^{\ell}a_j\mathpzc{d}_j(\frac{1}{r_j}-1)_+$, then
\begin{equation}\label{eq:thm:IR}
F^{s,\vec{a}}_{\vec{p},q}(\R^{\mathpzc{d}},\vec{w};X) = \bigcap_{j=1}^{\ell} L_{\check{\vec{p}}_{j}}\Big(\R^{\check{\mathpzc{d}}_{j}},\check{\vec{w}}_{j};\F^{s/a_j}_{p_j,q}\big(\R^{\mathpzc{d}_j},w_j;L_{\hat{\vec{p}}_{j}}(\R^{\hat{\mathpzc{d}}_{j}},\hat{\vec{w}}_{j});X\big)\Big). 
\end{equation}
\end{theorem}
\begin{proof}
This follows from
\cite[Corollary~5.3]{Lindemulder2019_IR} (and \cite[Examples~3.5$\&$3.20]{Lindemulder2019_IR}), cf.\ \cite[Example~5.8]{Lindemulder2019_IR}.
\end{proof}

Next we formulate a difference norm characterization for the space $\F^{s}_{p,q}(\R^d,w;E;X)$. As we will only need it for the power weight $w=w_\gamma$ in the one-dimensional setting, we restrict ourselves to that specific case. The advantage of this is that the conditions on the parameters become more explicit.

For a function $f$ on $\R$, $m \in \N$ and $h \in \R$, we define the $m$-th order difference $\Delta^m_hf$ of $f$ in the direction $h$ by 
\begin{equation*}
\Delta^{m}_{h}f(x) = \sum_{k=0}^{m}(-1)^{k}{m \choose k}f(x+kh), \qquad x \in \R.    
\end{equation*}
Furthermore, for $u \in (0,\infty)$ we define the mean
\begin{equation*}
d^{m}_{t,u}f(x,\varsigma)
= \left(\frac{1}{t}\int_{-t}^{t}\norm{\Delta^{m}_{h}f(x,\varsigma)}_{X}^{u}\ud h\right)^{1/u}.  
\end{equation*}

\begin{theorem}[\cite{Lindemulder2019_IR}]\label{thm:difference_norm}
Let $X$ be a Banach space, $E$ a quasi-Banach function space,
$s \in (0,\infty)$, $p \in (0,\infty)$, $q \in (0,\infty]$ and $\gamma \in (-1,\infty)$. Suppose that $E^{[r]}$ is a UMD Banach function space for some $r \in (0,\infty)$ and put
$$
r_E = \sup\Big\{ r \in (0,\infty) : E^{[r]} \text{is a UMD Banach function space} \Big\}.
$$
Suppose that $s > \max\{\frac{1}{q},\frac{1}{p},\frac{1+\gamma}{p},\frac{1}{r_E}\}-1$. Let $u \in (0,\infty)$ and $m \in \N$ be such that $s>\max\{\frac{1}{q},\frac{1}{p},\frac{1+\gamma}{p},\frac{1}{r_E}\}-\frac{1}{u}$ and $m > s$. Then, for all $f \in L_{1,\loc}(\R;E(X))$,
\begin{equation}\label{eq:thm:difference_norm}
\norm{f}_{\F^{s}_{p,q}(\R,w_\gamma;E;X)} \eqsim \norm{f}_{L_p(\R,w_\gamma;E(X))} +
\normB{\left(\int_0^1 [t^{-s}d^{m}_{t,u}f]^q \frac{\mathrm{d}t}{t} \right)^{1/q}}_{L_p(\R,w_\gamma;E(X))}
\end{equation}
\end{theorem}
\begin{proof}
Note that 
$$
\max\Big\{\frac{1}{q},\frac{1}{p},\frac{1+\gamma}{p},\frac{1}{r_E}\Big\} =
\inf\left\{\frac{1}{r} : r \in (0,p \wedge q), E^{[r]} \text{is a UMD Banach function space}, w_\gamma \in A_{p/r}  \right\}.
$$
The discrete version of \eqref{eq:thm:difference_norm} thus follows from
\cite[Corollary~4.7]{Lindemulder2019_IR} (and \cite[Examples~3.5$\&$3.20]{Lindemulder2019_IR}): 
for all $f \in L_{1,\loc}(\R,E(X))$,
\begin{equation*}
\norm{f}_{\F^{s}_{p,q}(\R,w_\gamma;E;X)} \eqsim \norm{f}_{L_p(\R,w_\gamma;E(X))} +
\normB{\left(\sum_{n \in \N} [2^{ns}d^{m}_{2^{-n},u}f]^q  \right)^{1/q}}_{L_p(\R,w_\gamma;E(X))}.
\end{equation*}
The equivalence \eqref{eq:thm:difference_norm} can be obtained from this by a standard monotonicty argument.
\end{proof}

\section{Pointwise Multiplication on Bessel Potential Spaces}\label{sec:pm;H}

In this subsection we will extend the multiplier property of $\one_{\R^d_+}$ on $H^{s}_{p}(\R^d,w_\gamma,X)$ from \cite[Theorem~1.1]{Meyries&Veraar2015_pointwise_multiplication} to the anisotropic mixed-norm setting. The approach in \cite{Meyries&Veraar2015_pointwise_multiplication} is based on the paraproduct technique as introduced by Bony (see e.g.\ \cite{Bony1981}) and a randomized Littlewood-Paley decomposition. In the context of pointwise multipliers, the use of paraproduct techniques goes back to Peetre \cite{Peetre1976} and Triebel \cite{Triebel1983_TFS_I} in order to treat Besov and Triebel-Lizorkin spaces in the full parameter range $p,q \in (0,\infty]$.

Alternative proofs of \cite[Theorem~1.1]{Meyries&Veraar2015_pointwise_multiplication} were afterwards obtained by the author \cite{Lindemulder2016_JFA} and by the author, Meyries and Veraar \cite{Lindemulder&Meyries&Veraar2017}. 
The latter proof is based on the representation of fractional powers of the negative Laplacian as a singular integral and the Hardy–Hilbert inequality and is much simpler than the proofs in~
\cite{Lindemulder2016_JFA,Meyries&Veraar2015_pointwise_multiplication}.

In order to extend \cite[Theorem~1.1]{Meyries&Veraar2015_pointwise_multiplication} to the anisotropic mixed-norm setting, we could try to extend one of its proofs from \cite{Lindemulder2016_JFA,Lindemulder&Meyries&Veraar2017,Meyries&Veraar2015_pointwise_multiplication}. However, we will not go into this direction. Instead, we will simply bootstrap the result itself to the anisotropic mixed-norm setting through the intersection representation Theorem~\ref{thm:IR;H} and the sum representation Corollary~\ref{cor:thm:IR;H;SR}.

For $i \in \{1,\ldots,\ell\}$ we write 
\begin{equation}\label{eq:half-space;decomp}
\R^\mathpzc{d}_{+,i} = [\R^{\mathpzc{d}_1} \times \cdots \R^{\mathpzc{d}_{i-1}}]  \times \R^{\mathpzc{d}_i}_+ \times [\R^{\mathpzc{d}_{i+1}} \times \cdots \R^{\mathpzc{d}_\ell}]
= \{ x \in \R^{\mathpzc{d}} : x_i \in \R^{\mathpzc{d}_i}_+\}.
\end{equation}

\begin{theorem}\label{thm:pm_H}
Let $X$ be a Banach space, $\vec{a} \in (\frac{1}{\N})^\ell$, $\vec{p} \in (1,\infty)^\ell$, $\vec{w} \in \prod_{j=1}^\ell A_{p_j}(\R^{\mathpzc{d}_j})$ and $s \in \R$. 
Let $i \in \{1,\ldots,\ell\}$ and suppose that $\mathpzc{d}_i=1$ and $w_i=w_\gamma$ for some $\gamma \in (-1,p_i-1)$.
If 
$$
a_i\left(\frac{1+\gamma}{p_i} - 1\right) < s < a_i\frac{1+\gamma}{p_i},
$$
then $\one_{\R^\mathpzc{d}_{+,i}}$ is a pointwise multiplier on the space $H^{s,\vec{a}}_{\vec{p}}(\R^{\mathpzc{d}},\vec{w};X)$ in the sense that
\begin{equation*}
\norm{\one_{\R^\mathpzc{d}_{+,i}}f}_{H^{s,\vec{a}}_{\vec{p}}(\R^{\mathpzc{d}},\vec{w};X)} \lesssim \norm{f}_{H^{s,\vec{a}}_{\vec{p}}(\R^{\mathpzc{d}},\vec{w};X)}, \qquad f \in H^{s,\vec{a}}_{\vec{p}}(\R^{\mathpzc{d}},\vec{w};X) \cap L_{\vec{p}}(\R^{\mathpzc{d}},\vec{w};X).    
\end{equation*}
\end{theorem}
\begin{proof}
Note that the case $s=0$ is trivial. So we only need to treat the case $s \neq 0$. In the notation of Theorem~\ref{thm:IR;H} and Corollary~\ref{cor:thm:IR;H;SR}, it trivially holds that 
\begin{equation}\label{eq:thm:pm_H;j}
\one_{\R^\mathpzc{d}_{+,i}} \in \mathcal{L}\left(
L_{\check{\vec{p}}_{j}}\Big(\R^{\check{\mathpzc{d}}_{j}},\check{\vec{w}}_{j};H^{s/a_j}_{p_j}\big(\R^{\mathpzc{d}_j},w_j;L_{\hat{\vec{p}}_{j}}(\R^{\hat{\mathpzc{d}}_{j}},\hat{\vec{w}}_{j};X)\big)\Big)\right)    
\end{equation}
for each $j \neq i$, as a pointwise multiplier. By Theorem~\ref{thm:IR;H} or  Corollary~\ref{cor:thm:IR;H;SR}, depending on whether $s>0$ or $s<0$, respectively, it thus suffices that \eqref{eq:thm:pm_H;j} is also valid for $j=i$. Denoting $Y=L_{\hat{\vec{p}}_{i}}(\R^{\hat{\mathpzc{d}}_{i}},\hat{\vec{w}}_{i};X)$, the latter is equivalent to $\one_{\R_+}$ being a pointwise multiplier on $H^{s/a_i}_{p_i}(\R,w_\gamma;Y)$. As the UMD property of $X$ is inherited by $Y$ (see \cite[Proposition~4.2.15]{Hytonen&Neerven&Veraar&Weis2016_Analyis_in_Banach_Spaces_I}) and as $\frac{1+\gamma}{p_i}-1<\frac{s}{a_i}<\frac{1+\gamma}{p_i}$, \cite[Theorem~1.1]{Meyries&Veraar2015_pointwise_multiplication} yields that $\one_{\R_+}$ is indeed a pointwise multiplier on $H^{s/a_i}_{p_i}(\R,w_\gamma;Y)$, as desired.
\end{proof}

\section{Pointwise Multiplication on Triebel-Lizorkin spaces}\label{sec:pm;F}

The following theorem is an analogue of Theorem~\ref{thm:pm_H} for the Triebel-Lizorkin scale, also see Corollary~\ref{cor:thm:pm_F} and Remark~\ref{rmk:cor:thm:pm_F;H} below. In its formulation we use the notation $\R^\mathpzc{d}_{+,i}$ from \eqref{eq:half-space;decomp}.

\begin{theorem}\label{thm:pm_F}
Let $X$ be a Banach space, $\vec{a} \in (0,\infty)^\ell$, $\vec{p} \in (0,\infty)^\ell$, $q \in (0,\infty]$, $\vec{w} \in \prod_{j=1}^\ell A_\infty(\R^{\mathpzc{d}_j})$ and $s \in (0,\infty)$. Let $\vec{r} \in (0,\infty)^{\ell}$ be such that $r_{j} < p_{1} \wedge \ldots \wedge p_{j} \wedge q$ for $j=1,\ldots,\ell$ and $\vec{w} \in \prod_{j=1}^{\ell}A_{p_{j}/r_{j}}(\R^{\mathpzc{d}_{j}})$.
Let $i \in \{1,\ldots,\ell\}$ and suppose that $\mathpzc{d}_i=1$ and $w_i=w_\gamma$ for some $\gamma \in (-1,\infty)$.
If 
\begin{equation}\label{eq:thm:pm_F}
a_i\left( \max\big\{\frac{1}{q},\frac{1}{p_1},\ldots,\frac{1}{p_i},\frac{1+\gamma}{p_i}\big\} - 1\right)_+ \,+\, \sum_{j=1,\ldots,\ell; j \neq i}a_j\mathpzc{d}_j\big(\frac{1}{r_j}-1\big)_+ < s < a_i\frac{1+\gamma}{p_i},    
\end{equation}
then $\one_{\R^\mathpzc{d}_{+,i}}$ is a pointwise multiplier on the space $F^{s,\vec{a}}_{\vec{p},q}(\R^{\mathpzc{d}},\vec{w};X) \subset L_{1,\loc}(\R^d;X)$.
\end{theorem}

\begin{remark}\label{rmk:thm:pm_F}
Note that, if $\vec{p} \in (1,\infty)^\ell$, $q \in [1,\infty]$ and $w_j \in A_{p_j}$ for each $j \neq i$, then each $r_j$ with $j \neq i$ can be taken arbitrarily close to $1$ and the condition \eqref{eq:thm:pm_F} reduces to
\begin{equation}\label{eq:rmk:thm:pm_F}
a_i\left( \frac{1+\gamma}{p_i} - 1\right)_+ < s < a_i\frac{1+\gamma}{p_i}.    
\end{equation}
\end{remark}

\begin{remark}\label{rmk:thm:pm_F;isotrop}
Note that, if $\ell=2$, $\vec{a}=(1,1)$, $\vec{p}=(p,p)$, $i=1$ and $w_2=\one_{\R^{\mathpzc{d}_2}}$, then $r_2$ can be taken arbitrarily close to $q \wedge p \wedge \frac{p}{1+\gamma}$ and condition \eqref{eq:thm:pm_F} reduces to
\begin{equation}\label{eq:rmk:thm:pm_F;unweight}
\left( \max\big\{\frac{1}{q},\frac{1}{p},\frac{1+\gamma}{p}\big\} - 1\right)_+ \,+\, (d-1)\left( \max\big\{\frac{1}{q},\frac{1}{p}\big\} - 1\right)_+  < s < \frac{1+\gamma}{p}. 
\end{equation}
As a consequence, if $s$ satisfies \eqref{eq:rmk:thm:pm_F;unweight}, $\one_{\R^d_+}$ is a pointwise multiplier on $F^s_{p,q}(\R^d,w_\gamma;X) \subset L_{1,\loc}(\R^d;X)$. For $X=\C$, $p \in (0,1)$, $\gamma=0$ this gives back the case $p \in (0,1)$ of \cite[Theorem~2.8.7]{Triebel1983_TFS_I}.
\end{remark}

\begin{remark}\label{rmk:thm:pm_F;B}
The corresponding version of Theorem~\ref{thm:pm_F} (and Corollary~\ref{cor:thm:pm_F} below) for Besov spaces could be derived by means of real interpolation. However, to the best of the author's knowledge, real interpolation results for weighted Triebel-Lizorkin spaces seem to be only present in the literature in the scalar-valued isotropic setting, see~\cite{Bui_Weighted_Beov&Triebel-Lizorkin_spaces:interpolation...,Rychkov2001}. 
\end{remark}

In the following corollary we remove the positive part from the left-hand side of \eqref{eq:rmk:thm:pm_F}, which yields an extension of the $F$-case of \cite[Theorem~1.3]{Meyries&Veraar2015_pointwise_multiplication} to the anisotropic setting with $\gamma$ beyond the $A_p$-range $(-1,p-1)$. 

\begin{cor}\label{cor:thm:pm_F}
Let $X$ be a Banach space, $\vec{a} \in (0,\infty)^\ell$, $\vec{p} \in (1,\infty)^\ell$, $q \in [1,\infty]$ and $s \in \R$. 
Let $i \in \{1,\ldots,\ell\}$ and suppose that $\mathpzc{d}_i=1$ and let $\vec{w}=(w_1,\ldots,w_\ell)$ be a weight vector with $w_i=w_\gamma$ for some $\gamma \in (-1,\infty)$ and $w_j \in A_{p_j}$ for each $j \neq i$.
If 
\begin{equation}\label{eq:cor:thm:pm_F}
a_i\left( \frac{1+\gamma}{p_i} - 1\right) < s < a_i\frac{1+\gamma}{p_i},
\end{equation}
then $\one_{\R^\mathpzc{d}_{+,i}}$ is a pointwise multiplier on the space $F^{s,\vec{a}}_{\vec{p},q}(\R^{\mathpzc{d}},\vec{w};X)$ in the sense that there is the estimate
\begin{equation}\label{eq:cor:thm:pm_F;est}
\norm{\one_{\R^\mathpzc{d}_{+,i}}f}_{F^{s,\vec{a}}_{\vec{p},q}(\R^{\mathpzc{d}},\vec{w};X)}    
\lesssim
\norm{f}_{F^{s,\vec{a}}_{\vec{p},q}(\R^{\mathpzc{d}},\vec{w};X)},
\qquad\qquad f \in \mathcal{S}(\R^d;X).
\end{equation}
\end{cor}

\begin{remark}\label{rmk:thm:pm_F;interpolation}
In the special case that $X$ is a UMD Banach space, $\vec{a} \in (\frac{1}{\N})^\ell$, $\vec{p} \in (1,\infty)^\ell$, $q \in (1,\infty)$, $w_j \in A_{p_j}$ for each $j \neq i$ and $\g \in (-1,p_i-1)$, the statement of Corollary~\ref{cor:thm:pm_F} can be obtained from Theorem~\ref{thm:pm_H} by means of $\ell^q$-interpolation, see \cite{LL20}.
\end{remark}

\begin{remark}\label{rmk:cor:thm:pm_F;H}
Theorem~\ref{thm:pm_F} and Corollary~\ref{cor:thm:pm_F} have corresponding versions for $\F^{s,\vec{a}}_{\vec{p},q}(\R^{\mathpzc{d}},\vec{w};E;X)$ in place of $F^{s,\vec{a}}_{\vec{p},q}(\R^{\mathpzc{d}},\vec{w};X)$ with the same proof, where $E$ is a Banach function space as in Lemma~\ref{lem:thm:pm_F,1-d} below. For this we need to mention that the intersection representation from Theorem~\ref{thm:IR}, which is one of the main ingredients in the proof, has a  corresponding version as well (see the references provided for Theorem~\ref{thm:IR}).

In view of the identity
$H^{s,\vec{a}}_{\vec{p}}(\R^{\mathpzc{d}},\vec{w};E(H)) = \F^{s,\vec{a}}_{\vec{p},2}(\R^{\mathpzc{d}},\vec{w};E;H)$ for $E$ a UMD Banach function space and $H$ a Hilbert space (see \cite{LL20}) this provides an alternative approach to Theorem~\ref{thm:pm_H} when $X$ is of the special form $X=E(H)$ or, more generally, when $X$ is isomorphic to a closed subspace of such a space $E(H)$. For many applications this restriction is actually no problem.
\end{remark}

\begin{proof}[Proof of Corollary~\ref{cor:thm:pm_F}]
If $\gamma \geq p_i-1$, then the conditions \eqref{eq:rmk:thm:pm_F} and \eqref{eq:cor:thm:pm_F} coincide. This case thus follows directly from Theorem~\ref{thm:pm_F} (through Remark~\ref{rmk:thm:pm_F}). So we may assume that $\gamma \in (-1,p_i-1)$, which just means that $w_\gamma \in A_{p_i}$. By Theorem~\ref{thm:pm_F} (and Remark~\ref{rmk:thm:pm_F}) and complex interpolation (see Proposition~\ref{prop:complex_interpolation}), it furthermore suffices to consider the case $s<0$.

Pick $\sigma \in (0,a_i\frac{1+\gamma}{p_i})$. Then $\one_{\R^\mathpzc{d}_{+,i}}$ is a pointwise multiplier on $F^{\sigma,\vec{a}}_{\vec{p},q}(\R^{\mathpzc{d}},\vec{w};X) \subset L_{1,\loc}(\R;X)$ by Theorem~\ref{thm:pm_F} (and Remark~\ref{rmk:thm:pm_F}). 
As $\mathcal{S}(\R^d;X) \subset F^{\sigma,\vec{a}}_{\vec{p},q}(\R^{\mathpzc{d}},\vec{w};X) \subset F^{s,\vec{a}}_{\vec{p},q}(\R^{\mathpzc{d}},\vec{w};X)$, we find that
\begin{equation}\label{eq:cor:thm:pm_F;est;comp_supp;contained}
\one_{\R^\mathpzc{d}_{+,i}}f \in F^{s,\vec{a}}_{\vec{p},q}(\R^{\mathpzc{d}},\vec{w};X),
\qquad\qquad f \in \mathcal{S}(\R^d;X).
\end{equation}

Let us first consider the case $q<\infty$. 
In this case we can apply Proposition~\ref{prop:prelim:duality} to obtain \eqref{eq:cor:thm:pm_F;dual;q-finite} and \eqref{eq:cor:thm:pm_F;dual;q-finite;norming}.
As $w_j' \in A_{p'_j}$ for each $j \neq i$, $w'_i=w_{\gamma'_p}$ with $\gamma'_p = -\frac{\gamma}{p_i-1} \in (-1,p_i'-1)$ and 
$$
0< -s < -a_i\left( \frac{1+\gamma}{p_i} - 1\right) = a_i\frac{1+\gamma'_p}{p'_i},
$$
by Theorem~\ref{thm:pm_F} (and Remark~\ref{rmk:thm:pm_F}) we obtain that $\one_{\R^\mathpzc{d}_{+,i}}$ is a pointwise multiplier on the function space $F^{-s,\vec{a}}_{\vec{p}',q'}(\R^{\mathpzc{d}},\vec{w}'_{\vec{p}};X^{*}) \subset L_{1,\loc}(\R^d;X^*)$, where the notation is as in Proposition~\ref{prop:prelim:duality}.

Now, in order to derive \eqref{eq:cor:thm:pm_F;est} by a duality argument, fix $f \in \mathcal{S}(\R^d;X)$. Then, 
by~\eqref{eq:cor:thm:pm_F;est;comp_supp;contained}, $\one_{\R^\mathpzc{d}_{+,i}}f \in F^{s,\vec{a}}_{\vec{p},q}(\R^{\mathpzc{d}},\vec{w};X)$.
Therefore, in view of \eqref{eq:cor:thm:pm_F;dual;q-finite;norming}, there exists $g \in \mathcal{S}(\R^d;X)$ with $\norm{g}_{F^{-s,\vec{a}}_{\vec{p}',q'}(\R^{\mathpzc{d}},\vec{w}'_{\vec{p}};X^{*})} \leq 1$ such that 
\begin{equation*}\label{eq:cor:thm:pm_F;est;norming;q_finite}
\norm{\one_{\R^\mathpzc{d}_{+,i}}f}_{F^{s,\vec{a}}_{\vec{p},q}(\R^{\mathpzc{d}},\vec{w};X)} \lesssim |\ip{\one_{\R^\mathpzc{d}_{+,i}}f}{g}|,  
\end{equation*}
where 
\begin{equation}\label{eq:cor:thm:pm_F;comp_adjoint}
\begin{aligned}
\ip{\one_{\R^\mathpzc{d}_{+,i}}f}{g} &\stackrel{\eqref{eq:cor:thm:pm_F;dual;q-finite;norming;pairing}}{=} 
\ip{\one_{\R^\mathpzc{d}_{+,i}}f}{g}_{\ip{\mathcal{S}'(\R^d;X)}{\mathcal{S}(\R^d;X^*)}} =
\int_{\R^d}\ip{\one_{\R^\mathpzc{d}_{+,i}}f(x)}{g(x)}_{\ip{X}{X^*}}\mathrm{d}x \\
&= \int_{\R^d}\ip{f(x)}{\one_{\R^\mathpzc{d}_{+,i}}g(x)}_{\ip{X}{X^*}}\mathrm{d}x 
= \ip{f}{\one_{\R^\mathpzc{d}_{+,i}}g}_{\ip{\mathcal{S}(\R^d;X)}{\mathcal{S}'(\R^d;X^*)}} \\
&= \ip{f}{\one_{\R^\mathpzc{d}_{+,i}}g}_{\ip{F^{s,\vec{a}}_{\vec{p},q}(\R^{\mathpzc{d}},\vec{w};X)}{F^{-s,\vec{a}}_{\vec{p}',q'}(\R^{\mathpzc{d}},\vec{w}'_{\vec{p}};X^{*})}}.
\end{aligned}
\end{equation}
Therefore,
$$
\norm{\one_{\R^\mathpzc{d}_{+,i}}f}_{F^{s,\vec{a}}_{\vec{p},q}(\R^{\mathpzc{d}},\vec{w};X)} \lesssim \norm{f}_{F^{s,\vec{a}}_{\vec{p},q}(\R^{\mathpzc{d}},\vec{w};X)}\norm{\one_{\R^\mathpzc{d}_{+,i}}g}_{F^{-s,\vec{a}}_{\vec{p}',q'}(\R^{\mathpzc{d}},\vec{w}'_{\vec{p}};X^{*})} \lesssim \norm{f}_{F^{s,\vec{a}}_{\vec{p},q}(\R^{\mathpzc{d}},\vec{w};X)}.
$$

Let us next consider the case $q=\infty$. By Proposition~\ref{prop:prelim:duality},
\begin{equation*}
F^{s,\vec{a}}_{\vec{p},\infty}(\R^{\mathpzc{d}},\vec{w};X^{**}) = [F^{-s,\vec{a}}_{\vec{p}',1}(\R^{\mathpzc{d}},\vec{w}'_{\vec{p}};X^*)]^*
\end{equation*}
under the natural pairing (induced by $\mathcal{S}'(\R^d;X^{**}) = [\mathcal{S}(\R^d;X^*)]'$). Similarly as in the case $q<\infty$, we have that $\one_{\R^\mathpzc{d}_{+,i}}$ is a pointwise multiplier on the function space $F^{-s,\vec{a}}_{\vec{p}',q'}(\R^{\mathpzc{d}},\vec{w}'_{\vec{p}};X^{*}) \subset L_{1,\loc}(\R^d;X^*)$.
Denoting this multiplication operator by $T$, its adjoint operator $T^* \in \mathcal{L}( F^{s,\vec{a}}_{\vec{p},\infty}(\R^{\mathpzc{d}},\vec{w};X^{**}))$ is given by $T^*f=\one_{\R^\mathpzc{d}_{+,i}}f$ for $f \in \mathcal{S}(\R^d;X^{**})$, which can be seen similarly to the computation \eqref{eq:cor:thm:pm_F;comp_adjoint}.
Since $\one_{\R^\mathpzc{d}_{+,i}}f \in \mathcal{S}'(\R^d;X)$ for $f \in \mathcal{S}(\R^d;X)$ and since $F^{s,\vec{a}}_{\vec{p},\infty}(\R^{\mathpzc{d}},\vec{w};X) = F^{s,\vec{a}}_{\vec{p},\infty}(\R^{\mathpzc{d}},\vec{w};X^{**}) \cap \mathcal{S}'(\R^{d};X)$ with an equality of norms, it follows that \eqref{eq:cor:thm:pm_F;est} holds true.
\end{proof}

We will use the intersection representation from Theorem~\ref{thm:IR} to reduce Theorem~\ref{thm:pm_F} to the one-dimensional setting. The following lemma treats the corresponding one-dimensional case.

\begin{lemma}\label{lem:thm:pm_F,1-d}
Let $X$ be a Banach space, $E$ a quasi-Banach function space, $p \in (0,\infty)$, $q \in (0,\infty]$, $\gamma \in (-1,\infty)$ and $s \in (0,\infty)$.
Suppose that $E^{[r]}$ is a UMD Banach function space for some $r \in (0,\infty)$ and put
\begin{equation}\label{eq:lem:thm:pm_F,1-d;E}
r_E = \sup\Big\{ r \in (0,\infty) : E^{[r]} \text{is a UMD Banach function space} \Big\}.    
\end{equation}
If $\max\{\frac{1}{q},\frac{1}{p},\frac{1+\gamma}{p},\frac{1}{r_E}\} - 1 < s < \frac{1+\gamma}{p}$, then $\one_{\R_+}$ is a pointwise multiplier on the space
$\F^{s}_{p,q}(\R,w_{\gamma};E;X) \subset L_{1,\loc}(\R;E(X))$
.
\end{lemma}

In the proof of this lemma we will use a description of $\F^{s}_{p,q}(\R,w_{\gamma};E;X)$ in terms of differences and a Hardy inequality. Such a strategy was used before in connection to the pointwise multiplier property of $1_{\R^d_+}$ by Strichartz \cite{Strichartz1967} for scalar-valued Bessel potential spaces $H^s_p(\R^d)$, by Triebel \cite[Section~2.8.6]{Triebel1983_TFS_I} for scalar-valued Triebel-Lizorkin spaces $F^{s}_{p,q}(\R^d)$ and by the author \cite[Theorem~1.4]{Lindemulder2016_JFA} for weighted vector-valued Bessel potential spaces $H^s_p(\R^d,w;X)$.

\begin{proof}
For convenience of notation we only treat the case $q < \infty$, the case $q=\infty$ being completely similar (with some minor modifications). Let $u \in (0,\infty)$ and $m \in \N$ be such that $\frac{1}{u}>\max\{\frac{1}{q},\frac{1}{p},\frac{1+\gamma}{p},\frac{1}{r_E}\}$ and $m > s$. Then the conditions of Theorem~\ref{thm:difference_norm} are fulfilled and we thus have \eqref{eq:thm:difference_norm} at our disposal. In order to use \eqref{eq:thm:difference_norm}, it will be convenient to introduce the notation
$$
[f]_{\F^{s}_{p,q}(\R,w_{\gamma};E;X)} := \normB{\left(\int_0^1 [t^{-s}d^{m}_{t,u}f]^q \frac{\mathrm{d}t}{t} \right)^{1/q}}_{L_p(\R,w_\gamma;E(X))}, \qquad f \in L_{1,\loc}(\R,E(X)).
$$

Fix  $f \in \F^{s}_{p,q}(\R,w_{\gamma};E;X) \subset L^{1}_{loc}(\R;E(X))$ and put $g:=1_{\R_{+}}f \in L_{1,\loc}(\R;E(X))$.

Let $(x,\varsigma) \in \R_{+} \times S$. Then
\begin{eqnarray*}
d^{m}_{t,u}g(x,\varsigma)
&=& \left(\frac{1}{t}\int_{-t}^{t}\norm{\Delta^{m}_{h}g(x,\varsigma)}_{X}^{u}dh\right)^{1/u} \\
&\lesssim& \left(\frac{1}{t}\int_{\max\{-x/m,-t\}}^{t}\norm{\Delta^{m}_{h}g(x,\varsigma)}_{X}^{u}dh\right)^{1/u}
      +  \left(\frac{1}{t}\int^{\max\{-x/m,-t\}}_{-t}\norm{\Delta^{m}_{h}g(x,\varsigma)}_{X}^{u}dh\right)^{1/u}  \\
&=& \left(\frac{1}{t}\int_{\max\{-x/m,-t\}}^{t}\norm{\Delta^{m}_{h}f(x,\varsigma)}_{X}^{u}dh\right)^{1/u}
      +  \left(\frac{1}{t}\int^{\max\{-x/m,-t\}}_{-t}\norm{\sum_{j=0}^{m-1}(-1)^{m-j}{m \choose j}g(x+jh,\varsigma)}_{X}^{u}dh\right)^{1/u}  \\
&\lesssim& \left(\frac{1}{t}\int_{\max\{-x/m,-t\}}^{t}\norm{\Delta^{m}_{h}f(x,\varsigma)}_{X}^{u}dh\right)^{1/u}
      +  \sum_{j=0}^{m-1}\left(\frac{1}{t}\int^{\max\{-x/m,-t\}}_{-t}\norm{f(x+jh,\varsigma)}_{X}^{u}dh\right)^{1/u}  \\
&\eqsim& d^{m}_{t,u}f(x,\varsigma) + 1_{\{t>x/m\}}t^{-1/u}(t-x/m)^{1/u}\,\sum_{j=0}^{m-1}\left( \fint_{-t}^{-x/m}\norm{f(x+jh,\varsigma)}_{X}^{u}dh \right)^{1/u} \\
&& d^{m}_{t,u}f(x,\varsigma) + 1_{\{t>x/m\}}t^{-1/u}(t-x/m)^{1/u}\,\left( \norm{f(x,\varsigma)}_{X} + \sum_{j=1}^{m-1}\left(\fint_{-jt}^{-jx/m}\norm{f(x+y,\varsigma)}_{X}^{u}dy \right)^{1/u} \right) \\
&\lesssim& d^{m}_{t,u}f(x,\varsigma) + 1_{\{t>x/m\}}t^{-1/u}(t-x/m)^{1/u}\,\left( \norm{f(x,\varsigma)}_{X} + \sum_{j=1}^{m-1}M_{u}\norm{f}_{X}\left(\frac{m-j}{m}x\right) \right),
\end{eqnarray*}
from which it follows that
\begin{eqnarray*}
 \left( \int_{0}^{1} [t^{-s}d^{m}_{t,u}g(x,\varsigma)]
 ^{q} \frac{\ud t}{t} \right)^{1/q}
&\lesssim& \left( \int_{0}^{1} t^{-sq}d^{m}_{t,u}f(x,\varsigma)^{q} \frac{\ud t}{t} \right)^{1/q} \: + \:
\left( \int_{x/m}^{\infty}t^{-sq}t^{-q/u}(t-x/m)^{q}\frac{\ud t}{t} \right)^{1/q} \\
&& \quad \cdot \: \left[ \norm{f(x,\varsigma)}_{X} + \sum_{j=1}^{m-1} M_{u}\norm{f}_{X}\left(\frac{m-j}{m}x\right) \right] \\
&\lesssim& \left( \int_{0}^{1} t^{-sq}d^{m}_{t,u}f(x,\varsigma)^{q} \frac{\ud t}{t} \right)^{1/q} \\
&& \quad |x|^{-s}\left[ \norm{f(x,\varsigma)}_{X} + \sum_{j=1}^{m-1} M_{u}\norm{f}_{X}\left(\frac{m-j}{m}x\right) \right],
\end{eqnarray*}
where we used
\[
\int_{a}^{\infty}t^{\alpha}(t-a)^{\beta}\ud t = a^{\alpha+\beta+1}\int_{1}^{\infty}\tau^{\alpha}(\tau-1)^{\beta}d\tau \lesssim_{\alpha,\beta} a^{\alpha+\beta+1}, \quad\quad \alpha \in \R, \beta \geq 0, \alpha+\beta<-1,
\]
with $\alpha=-sq-q/u-1$ and $\beta=q/u$.
Taking $L^{p}(\R_{+},w_{\gamma})[E]$-norms subsequently yields
\begin{eqnarray*}
\normB{  \left( \int_{0}^{1} [t^{-s}d^{m}_{t,u}g]^{q} \frac{\ud t}{t} \right)^{1/q} }_{L^{p}(\R_{+},w_{\gamma})[E]}
&\lesssim& [f]_{\F^{s}_{p,q}(\R,w_{\gamma};E;X)} \: + \: \norm{f}_{L^{p}(\R,w_{-sp+\gamma};E(X))} \\
&& \quad + \: \sum_{j=1}^{m-1}\normB{ (x,\varsigma) \mapsto M_{u}\norm{f}_{X}\big(\frac{m-j}{m}x,\varsigma\big) }_{L^{p}(\R,w_{-sp+\gamma})[E]} \\
&\lesssim& [f]_{\F^{s}_{p,q}(\R,w_{\gamma};E;X)} \: + \: \norm{f}_{L^{p}(\R,w_{-sp+\gamma};E(X))} \\
&& \quad + \: \sum_{j=1}^{m-1}\normb{ M_{u}\norm{f}_{X} }_{L^{p}(\R,w_{-sp+\gamma})[E]} \\
&\lesssim& [f]_{\F^{s}_{p,q}(\R,w_{\gamma};E;X)} \: + \: \norm{f}_{L^{p}(\R,w_{-sp+\gamma};E(X))},
\end{eqnarray*}
where we used that $M_{u}$ is bounded on $L^{p}(\R,w_{-sp+\gamma})[E]$ as $E^{[u]}$ is a UMD Banach function space and $w_{-sp+\gamma} \in A_{p/u}$ (which is in fact equivalent to $\frac{1+\gamma}{p}-\frac{1}{u} < s < \frac{1+\gamma}{p}$).
Since
\[
\F^{s}_{p,q}(\R,w_{\gamma};E;X) \hookrightarrow F^{s}_{p,\infty}(\R,w_{\gamma};E(X)) \hookrightarrow F^{0}_{p,1}(\R,w_{-sp+\gamma};E(X))
\hookrightarrow L^{p}(\R,w_{-sp+\gamma};E(X)),
\]
where the second inclusion follows from \cite[Theorem~1.2]{Meyries&Veraar_sharp_embd_power_weights},
it follows that
\begin{equation}\label{eq:lem:thm:pm_F,1-d;R+}
\normB{  \left( \int_{0}^{\infty} [t^{-s}d^{m}_{t,u}(1_{\R_{+}}f)]^{q} \frac{\ud t}{t} \right)^{1/q} }_{L^{p}(\R_{+},w_{\gamma})[E]}
\lesssim [f]_{\F^{s}_{p,q}(\R,w_{\gamma};E;X)} + \norm{f}_{\F^{s}_{p,q}(\R,w_{\gamma};E;X)}.
\end{equation}
As
\[
d^{m}_{t,u}(1_{\R_{+}}f)(x,\varsigma) = d^{m}_{t,u}(f-1_{\R_{-}}f)(x,\varsigma) \lesssim
d^{m}_{t,u}(f)(x,\varsigma) + d^{m}_{t,u}(1_{\R_{+}}[f(-\,\cdot\,)])(-x,\varsigma),
\]
the above inequality \eqref{eq:lem:thm:pm_F,1-d;R+} is also valid with $L^{p}(\R_{+},w_{\gamma})[E]$ replaced by $L^{p}(\R_{-},w_{\gamma})[E]$.
Combining \eqref{eq:lem:thm:pm_F,1-d;R+} with this corresponding version on $\R_-$, we find
\begin{equation*}
[1_{\R_{+}}f]_{\F^{s}_{p,q}(\R,w_{\gamma};E;X)} \lesssim 
[f]_{\F^{s}_{p,q}(\R,w_{\gamma};E;X)} + \norm{f}_{\F^{s}_{p,q}(\R,w_{\gamma};E;X)}.    
\end{equation*}
In combination with \eqref{eq:thm:difference_norm}, this yields the desired estimate
$$
\norm{1_{\R_{+}}f}_{\F^{s}_{p,q}(\R,w_{\gamma};E;X)} \lesssim \norm{f}_{\F^{s}_{p,q}(\R,w_{\gamma};E;X)}.
$$
\end{proof}

\begin{proof}[Proof of Theorem~\ref{thm:pm_F}]
Note that
$$
\max\Big\{\frac{1}{q},\frac{1}{p_1},\ldots,\frac{1}{p_i},\frac{1+\gamma}{p_i}\Big\} = \inf\left\{ \frac{1}{r} : 0 < r < p_{1} \wedge \ldots \wedge p_{i} \wedge q, w_\gamma \in A_{p_i/r}  \right\}.
$$
Therefore, by choosing $r_i$ bigger if necessary, we may without loss of generality assume that the condition $s>\sum_{j=1}^{\ell}a_j\mathpzc{d}_j(\frac{1}{r_j}-1)_+$ is fulfilled.
Then, by Theorem~\ref{thm:IR}, the intersection representation \eqref{eq:thm:IR} is valid. Writing $E=L_{\hat{\vec{p}}_{i}}(\R^{\hat{\mathpzc{d}}_{i}},\hat{\vec{w}}_{i})$, it thus remains to be shown that $\one_{\R_+}$ is a pointwise multiplier on $\F^{s/a_i}_{p_i,q}(\R,w_\gamma;E;X)$. From the observation that
$r_E = p_1 \wedge \ldots \wedge p_{i-1}$, where $r_E$ is as in \eqref{eq:lem:thm:pm_F,1-d;E}, it follows that
$$
\max\Big\{\frac{1}{q},\frac{1}{p_i},\frac{1+\gamma}{p_i},\frac{1}{r_E}\Big\} = 
\max\Big\{\frac{1}{q},\frac{1}{p_1},\ldots,\frac{1}{p_i},\frac{1+\gamma}{p_i}\Big\}.
$$
As a consequence, $\max\{\frac{1}{q},\frac{1}{p},\frac{1+\gamma}{p},\frac{1}{r_E}\} - 1 < s < \frac{1+\gamma}{p}$. We can thus invoke Lemma~\ref{lem:thm:pm_F,1-d} to obtain that $\one_{\R_+}$ is indeed a pointwise multiplier on $\F^{s/a_i}_{p_i,q}(\R,w_\gamma;E;X)$. 
\end{proof}

\bibliographystyle{plain}

\def\cprime{$'$} \def\cprime{$'$} \def\cprime{$'$} \def\cprime{$'$}

\end{document}